\documentclass[reqno]{amsart}

\usepackage{amsmath,amssymb,amsthm,amsfonts,mathrsfs}
\usepackage{fullpage}
\usepackage{xcolor}
\usepackage{graphicx}
\usepackage{listings}
\usepackage{bbm}

\definecolor{codegreen}{rgb}{0,0.6,0}
\definecolor{codegray}{rgb}{0.5,0.5,0.5}
\definecolor{codepurple}{rgb}{0.58,0,0.82}
\definecolor{backcolour}{rgb}{0.95,0.95,0.92}

\lstdefinestyle{mystyle}{
  backgroundcolor=\color{backcolour},   commentstyle=\color{codegreen},
  keywordstyle=\color{magenta},
  numberstyle=\tiny\color{codegray},
  stringstyle=\color{codepurple},
  basicstyle=\ttfamily\footnotesize,
  breakatwhitespace=false,         
  breaklines=true,                 
  captionpos=b,                    
  keepspaces=true,                 
  numbers=left,                    
  numbersep=5pt,                  
  showspaces=false,                
  showstringspaces=false,
  showtabs=false,                  
  tabsize=2
}

\usepackage{listings}
\usepackage[colorlinks=true,
linkcolor=blue,
anchorcolor=blue,
citecolor=red
]{hyperref}
\allowdisplaybreaks 
\newtheorem{theorem}{Theorem}[section]
\newtheorem{lemma}[theorem]{Lemma}

\newtheorem*{conjecture*}{Conjecture}
\theoremstyle{definition}

\theoremstyle{remark}

\newtheorem*{remark*}{remark}

\usepackage{colonequals}

\author{Runbo Li}
\address{International Curriculum Center, The High School Affiliated to Renmin University of China, Beijing, China}
\email{runbo.li.carey@gmail.com}

\makeatletter
\@namedef{subjclassname@2020}{\textup{}2020 Mathematics Subject Classification}
\makeatother

\title[]{Primes in arithmetic progressions to smooth moduli: A minorant version}

\begin{document}
	
\begin{abstract}
The author prove that there exists a function $\rho(n)$ which is a minorant for the prime indicator function $\mathbbm{1}_{p}(n)$ and has distribution level $\frac{10}{19}$ in arithmetic progressions to smooth moduli. This refines the previous results of Baker--Irving and Stadlmann.
\end{abstract}

\maketitle

\tableofcontents

\section{Introduction}
One of the famous topics in number theory is to study the distribution of primes in arithmetic progressions. Given some $\theta > 0$, $A > 0$ and sets $\mathcal{Q}(x) \subseteq \mathbb{N}$, we are looking for results of the type
\begin{equation}
\sum_{\substack{q \leqslant x^{\theta - \varepsilon} \\ q \in \mathcal{Q}(x) \\ (q,a)=1 }} \left| \sum_{\substack{n \leqslant x \\ n \equiv a (\bmod q)}} \mathbbm{1}_{p}(n) - \frac{1}{\varphi(q)} \sum_{\substack{n \leqslant x \\ (n,q)=1}} \mathbbm{1}_{p}(n) \right| \ll \frac{x}{(\log x)^{A}}.
\end{equation}
When $\mathcal{Q}(x) = \mathbb{N}$, the most famous result is due to Bombieri \cite{BombieriBV} and Vinogradov \cite{VinogradovBV}, who showed in 1965 that (1) holds with $\theta = \frac{1}{2}$. The exponent $\frac{1}{2}$ is also the limit obtained under Generalized Riemann Hypothesis (GRH), Hence improving this result directly is extremely difficult.

Now we are focusing on the case $\mathcal{Q}(x) = \left\{q: q \in \mathbb{N}, q \mid \prod_{p < x^{\delta}} p \right\}$ or square-free $x^{\delta}$-smooth moduli. Then (1) may be written as
\begin{equation}
\sum_{\substack{q \leqslant x^{\theta - \varepsilon} \\ q \mid \prod_{p < x^{\delta}} p \\ (q,a)=1 }} \left| \sum_{\substack{n \leqslant x \\ n \equiv a (\bmod q)}} \mathbbm{1}_{p}(n) - \frac{1}{\varphi(q)} \sum_{\substack{n \leqslant x \\ (n,q)=1}} \mathbbm{1}_{p}(n) \right| \ll \frac{x}{(\log x)^{A}}
\end{equation}
in this case. This type of results have played an important role in the study of bounded gaps between primes, see \cite{ZhangYitang} \cite{Polymath8a}. In \cite{ZhangYitang} Zhang proved (2) holds with $\theta = \frac{1}{2} + \frac{1}{584} \approx 0.5017$, which was later improved by Polymath \cite{Polymath8a} to $\theta = \frac{1}{2} + \frac{7}{300} \approx 0.5233$ and by Stadlmann \cite{Stadlmann525} to $\theta = \frac{1}{2} + \frac{1}{40}=0.525$.

In 2017, Baker and Irving \cite{BakerIrving} considered a variant of (2). They constructed a minorant $\rho(n)$ for the prime indicator function $\mathbbm{1}_{p}(n)$ and proved corresponding result
\begin{equation}
\sum_{\substack{q \leqslant x^{\theta - \varepsilon} \\ q \mid \prod_{p < x^{\delta}} p \\ (q,a)=1 }} \left| \sum_{\substack{n \leqslant x \\ n \equiv a (\bmod q)}} \rho(n) - \frac{1}{\varphi(q)} \sum_{\substack{n \leqslant x \\ (n,q)=1}} \rho(n) \right| \ll \frac{x}{(\log x)^{A}}
\end{equation}
with $\theta = \frac{1}{2} + \frac{7}{300} + \frac{187}{197700} \approx 0.5243$. In their paper Harman's sieve \cite{HarmanBOOK} was used to construct a suitable minorant and prove stronger results on the length of bounded intervals containing many primes. Stadlmann \cite{Stadlmann525} further improved this to $\theta = 0.5253$, which is the current best distribution level in this direction.

In this paper, we shall use a delicate sieve decomposition to prove (3) with $\theta = \frac{10}{19} \approx 0.5263$. A defect of our method is that the lower bound of our minorant is much worse than it in \cite{BakerIrving} and \cite{Stadlmann525}. Both of them use the lower bounds (very close to $\mathbbm{1}_{p}(n)$) to handle the bounded prime gap problem. Our method leads to nothing new on that topic.
\begin{theorem}\label{t1}
There exists a function $\rho(n)$ which satisfies the following properties:

(Minorant) $\rho(n)$ is a minorant for the prime indicator function $\mathbbm{1}_{p}(n)$. That is, we have
\begin{align*}
\rho(n) \leqslant
\begin{cases}
1, & n \text{ is prime}, \\
0, & \text{otherwise}. \\
\end{cases}
\end{align*}

(No small prime factors) If $n$ has a prime factor less than some fixed $\xi > 0$, then $\rho(n) = 0$.

(Lower bound) We have
$$
\sum_{n \sim x} \rho(n) \geqslant 0.75 (1+o(1))\frac{x}{\log x}.
$$

(Distribution in Arithmetic Progressions to smooth moduli) For any integer $a$ that coprime to $\prod_{p < x^{\delta}} p$ and any $A>0$, we have
$$
\sum_{\substack{q \leqslant x^{\frac{10}{19} - \varepsilon} \\ q \mid \prod_{p < x^{\delta}} p \\ (q,a)=1 }} \left| \sum_{\substack{n \leqslant x \\ n \equiv a (\bmod q)}} \rho(n) - \frac{1}{\varphi(q)} \sum_{\substack{n \leqslant x \\ (n,q)=1}} \rho(n) \right| \ll \frac{x}{(\log x)^{A}}.
$$
\end{theorem}

Throughout this paper, we always suppose that $\delta = 10^{-100}$ and $x$ is sufficiently large. The letter $p$, with or without subscript, is reserved for prime numbers. We define the sieve function $\psi\left(n, z\right)$ as
\begin{align*}
\psi\left(n, z\right) = 
\begin{cases}
1, & \left(n, \prod_{p < z} p \right) = 1, \\
0, & \text{otherwise}.
\end{cases}
\end{align*}

\section{Asymptotic formulas}
Now we list the arithmetic information needed for the theorem. The definitions of ``the Siegel--Walfisz condition'' and ``smooth'' can be found in \cite{Polymath8a}.
\begin{lemma}\label{l21}
Suppose that a function $f: \mathbb{N} \to \mathbb{C}$ satisfies one of the following conditions:

(Type-I) $f = \alpha * \beta$ where $\alpha$ and $\beta$ are coefficient sequences at scales $M$ and $N$. Moreover, assume that $\alpha$ satisfies the Siegel--Walfisz condition, $\beta$ is smooth, $MN \asymp x$ and 
$$
N \geqslant x^{\frac{13}{38}};
$$

(Type-II) $f = \alpha * \beta$ where $\alpha$ and $\beta$ are coefficient sequences at scales $M$ and $N$. Moreover, assume that $\alpha$ and $\beta$ satisfy the Siegel--Walfisz condition, $MN \asymp x$ and 
$$
x^{\frac{8}{19}} \leqslant N \leqslant x^{\frac{11}{19}}.
$$

Then for any integer $a$ that coprime to $\prod_{p < x^{\delta}} p$ and any $A>0$, we have
$$
\sum_{\substack{q \leqslant x^{\frac{10}{19} - \varepsilon} \\ q \mid \prod_{p < x^{\delta}} p \\ (q,a)=1 }} \left| \sum_{\substack{n \leqslant x \\ n \equiv a (\bmod q)}} f(n) - \frac{1}{\varphi(q)} \sum_{\substack{n \leqslant x \\ (n,q)=1}} f(n) \right| \ll \frac{x}{(\log x)^{A}}.
$$
\end{lemma}

\begin{proof}
The proof is very similar to that of [\cite{StadlmannPhD}, Lemma 3.20 (I)(II)].
\end{proof}

By Lemma~\ref{l21}, we can easily deduce the following two lemmas.

\begin{lemma}\label{l22}
Let
$$
f(x) = \sum_{p_1, \ldots, p_n} \psi\left(\frac{n}{p_1 \cdots p_n}, x^{\frac{3}{19}}\right).
$$
Then Lemma~\ref{l21} holds for $f(n)$ if we can group $\{1, \cdots, n\}$ into $I$ and $J$ such that
$$
\prod_{i \in I}p_i \leqslant x^{\frac{8}{19}} \qquad \text{and} \qquad \prod_{j \in J}p_j \leqslant x^{\frac{9}{38}}.
$$
\end{lemma}

\begin{lemma}\label{l23}
Let
$$
f(x) = \sum_{p_1, \ldots, p_n} \psi\left(\frac{n}{p_1 \cdots p_n}, p_n\right).
$$
Then Lemma~\ref{l21} holds for $f(n)$ if we can group $\{1, \cdots, n\}$ into $I$ and $J$ such that
$$
x^{\frac{8}{19}} \leqslant \prod_{i \in I}p_i \leqslant x^{\frac{11}{19}}.
$$
\end{lemma}

Our aim is to decompose the prime indicator function $\mathbbm{1}_{p}(n)$ into sieve functions of the above forms and show that the total loss from the dropped parts (which don't satisfy the conditions in Lemma~\ref{l22} or Lemma~\ref{l23} and must be non-negative) is less than $1 - 0.75 = 0.25$ in order to get a positive lower bound with same order of magnitude.

\section{The final decomposition}
In this section we will decompose the prime indicator function $\mathbbm{1}_{p}(n)$ using Buchstab's identity. Let $\omega(u)$ denotes the Buchstab function determined by the following differential-difference equation
\begin{align*}
\begin{cases}
\omega(u)=\frac{1}{u}, & \quad 1 \leqslant u \leqslant 2, \\
(u \omega(u))^{\prime}= \omega(u-1), & \quad u \geqslant 2 .
\end{cases}
\end{align*}
Moreover, we have the upper and lower bounds for $\omega(u)$:
\begin{align*}
\omega(u) \geqslant \omega_{0}(u) =
\begin{cases}
\frac{1}{u}, & \quad 1 \leqslant u < 2, \\
\frac{1+\log(u-1)}{u}, & \quad 2 \leqslant u < 3, \\
\frac{1+\log(u-1)}{u} + \frac{1}{u} \int_{2}^{u-1}\frac{\log(t-1)}{t} d t \geqslant 0.5607, & \quad 3 \leqslant u < 4, \\
0.5612, & \quad u \geqslant 4, \\
\end{cases}
\end{align*}
\begin{align*}
\omega(u) \leqslant \omega_{1}(u) =
\begin{cases}
\frac{1}{u}, & \quad 1 \leqslant u < 2, \\
\frac{1+\log(u-1)}{u}, & \quad 2 \leqslant u < 3, \\
\frac{1+\log(u-1)}{u} + \frac{1}{u} \int_{2}^{u-1}\frac{\log(t-1)}{t} d t \leqslant 0.5644, & \quad 3 \leqslant u < 4, \\
0.5617, & \quad u \geqslant 4. \\
\end{cases}
\end{align*}
We shall use $\omega_0(u)$ and $\omega_1(u)$ to give numerical bounds for some sieve functions discussed below. Let $p_j = (2x)^{t_j}$ and by Buchstab's identity, we have
\begin{align}
\nonumber \mathbbm{1}_{p}(n) =&\ \psi\left(n, (2x)^{\frac{1}{2}}\right) \\
\nonumber =&\ \psi\left(n, x^{\frac{3}{19}}\right) - \sum_{\frac{3}{19} \leqslant t_1 < \frac{8}{19}} \psi\left(\frac{n}{p_1}, x^{\frac{3}{19}}\right) - \sum_{\frac{8}{19} \leqslant t_1 < \frac{1}{2}} \psi\left(\frac{n}{p_1}, p_1\right) \\
\nonumber & + \sum_{\substack{\frac{3}{19} \leqslant t_1 < \frac{8}{19} \\ \frac{3}{19} \leqslant t_2 < \min\left(t_1, \frac{1}{2}(1-t_1) \right) }} \psi\left(\frac{n}{p_1 p_2}, p_2\right) \\
=&\ S_{1} - S_{2} - S_{3} + S_{4}.
\end{align}
By Lemmas~\ref{l22}--\ref{l23} we know that Lemma~\ref{l21} holds for $S_1$--$S_3$, hence we only need to consider $S_4$. Before further decomposing, we define non-overlapping polygons $A$, $B$, $C$, $D$, whose union is 
$$
\left\{(t_1, t_2): \frac{3}{19} \leqslant t_1 < \frac{8}{19},\ \frac{3}{19} \leqslant t_2 < \min\left(t_1, \frac{1}{2}(1-t_1) \right),\ t_1+t_2 \notin \left[\frac{8}{19}, \frac{11}{19}\right] \right\}.
$$
These regions are defined as
\begin{align}
\nonumber A =&\ \left\{(t_1, t_2): \frac{3}{19} \leqslant t_1 < \frac{8}{19},\ \frac{3}{19} \leqslant t_2 < \min\left(t_1, \frac{1}{2}(1-t_1)\right),\ t_1+t_2 < \frac{8}{19} \right\}, \\
\nonumber B =&\ \left\{(t_1, t_2): \frac{3}{19} \leqslant t_1 < \frac{8}{19},\ \frac{3}{19} \leqslant t_2 < \min\left(t_1, \frac{1}{2}(1-t_1)\right),\ t_1+t_2 > \frac{11}{19},\ t_2 < \frac{9}{38} \right\}, \\
\nonumber C =&\ \left\{(t_1, t_2): \frac{3}{19} \leqslant t_1 < \frac{8}{19},\ \frac{3}{19} \leqslant t_2 < \min\left(t_1, \frac{1}{2}(1-t_1)\right),\ t_1+t_2 > \frac{11}{19},\ t_2 > \frac{9}{38} \right\}.
\end{align}
Now we have
\begin{align}
\nonumber S_{4} =&\ \sum_{(t_1, t_2) \in A} \psi\left(\frac{n}{p_1 p_2}, p_2\right) + \sum_{(t_1, t_2) \in B} \psi\left(\frac{n}{p_1 p_2}, p_2\right) + \sum_{(t_1, t_2) \in C} \psi\left(\frac{n}{p_1 p_2}, p_2\right) \\
=&\ S_{A} + S_{B} + S_{C}.
\end{align}
We first decompose $S_{A}$. By Buchstab's identity, we have
\begin{align}
\nonumber S_{A} =&\ \sum_{(t_1, t_2) \in A} \psi\left(\frac{n}{p_1 p_2}, p_2\right) \\
\nonumber =&\ \sum_{(t_1, t_2) \in A} \psi\left(\frac{n}{p_1 p_2}, x^{\frac{3}{19}}\right) - \sum_{\substack{(t_1, t_2) \in A \\ \frac{3}{19} \leqslant t_3 < \min\left(t_2, \frac{1}{2}(1-t_1-t_2) \right) }} \psi\left(\frac{n}{p_1 p_2 p_3}, x^{\frac{3}{19}}\right) \\
\nonumber & + \sum_{\substack{(t_1, t_2) \in A \\ \frac{3}{19} \leqslant t_3 < \min\left(t_2, \frac{1}{2}(1-t_1-t_2) \right) \\ \frac{3}{19} \leqslant t_4 < \min\left(t_3, \frac{1}{2}(1-t_1-t_2-t_3) \right) }} \psi\left(\frac{n}{p_1 p_2 p_3 p_4}, p_4\right) \\
=&\ S_{A1} - S_{A2} + S_{A3}.
\end{align}
We know that Lemma~\ref{l21} holds for $S_{A1}$. Since $t_1 + t_2 < \frac{8}{19}$ and $t_2 < t_1$, we have $t_3 < t_2 < \frac{1}{2}(t_1 + t_2) = \frac{4}{19} < \frac{9}{38}$, and Lemma~\ref{l21} also holds for $S_{A2}$. For $S_{A3}$, we can use Lemma~\ref{l23} to handle part of $S_{A3}$ if we can group $\{1,2,3,4\}$ into $I$ and $J$ satisfy the corresponding conditions. For the remaining part, we cannot ensure that it has a distribution level of $\frac{10}{19}$, hence we need to discard it. In this way we obtain a loss from $S_{A}$ of
\begin{equation}
\int_{(t_1, t_2, t_3, t_4) \in U_{A3}} \frac{\omega_1 \left(\frac{1 - t_1 - t_2 - t_3 - t_4}{t_4}\right)}{t_1 t_2 t_3 t_4^2} d t_4 d t_3 d t_2 d t_1 < 0.000829,
\end{equation}
where
\begin{align}
\nonumber U_{A3}(t_1, t_2, t_3, t_4) :=&\ \left\{ (t_1, t_2) \in A,\ \frac{3}{19} \leqslant t_3 < \min\left(t_2, \frac{1}{2}(1-t_1-t_2)\right), \right. \\
\nonumber & \quad \{1,2,3\} \text{ cannot be partitioned into } I \text{ and } J \text{ in Lemma~\ref{l23}}, \\
\nonumber & \quad \frac{3}{19} \leqslant t_4 < \min\left(t_3, \frac{1}{2}(1-t_1-t_2-t_3) \right), \\
\nonumber & \quad \{1,2,3,4\} \text{ cannot be partitioned into } I \text{ and } J \text{ in Lemma~\ref{l23}}, \\
\nonumber & \left. \quad \frac{3}{19} \leqslant t_1 < \frac{1}{2},\ \frac{3}{19} \leqslant t_2 < \min\left(t_1, \frac{1}{2}(1-t_1) \right) \right\}.
\end{align}

For $S_{B}$ we cannot perform a straightforward decomposition as in $S_{A}$. Nonetheless, we can perform a variable role-reversal since we have $t_1 < \frac{8}{19}$, $1-t_1-t_2 < \frac{8}{19}$ and $t_2 < \frac{9}{38}$. We refer the readers to \cite{LRB052} for more applications of role-reversals. By a similar process as in \cite{LRB052}, we have
\begin{align}
\nonumber S_{B} =&\ \sum_{(t_1, t_2) \in B} \psi\left(\frac{n}{p_1 p_2}, p_2\right) \\
\nonumber =&\ \sum_{(t_1, t_2) \in B} \psi\left(\frac{n}{p_1 p_2}, x^{\frac{3}{19}}\right) - \sum_{\substack{(t_1, t_2) \in B \\ \frac{3}{19} \leqslant t_3 < \min\left(t_2, \frac{1}{2}(1-t_1-t_2) \right) }} \psi\left(\frac{n}{p_1 p_2 p_3}, p_3\right) \\
\nonumber =&\ \sum_{(t_1, t_2) \in B} \psi\left(\frac{n}{p_1 p_2}, x^{\frac{3}{19}}\right) - \sum_{\substack{(t_1, t_2) \in B \\ \frac{3}{19} \leqslant t_3 < \min\left(t_2, \frac{1}{2}(1-t_1-t_2) \right) }} \psi\left(\frac{n}{\beta p_2 p_3}, \left(\frac{2x}{\beta p_2 p_3}\right)^{\frac{1}{2}} \right) \\
\nonumber =&\ \sum_{(t_1, t_2) \in B} \psi\left(\frac{n}{p_1 p_2}, x^{\frac{3}{19}}\right) - \sum_{\substack{(t_1, t_2) \in B \\ \frac{3}{19} \leqslant t_3 < \min\left(t_2, \frac{1}{2}(1-t_1-t_2) \right) }} \psi\left(\frac{n}{\beta p_2 p_3}, x^{\frac{3}{19}} \right) \\
\nonumber & + \sum_{\substack{(t_1, t_2) \in B \\ \frac{3}{19} \leqslant t_3 < \min\left(t_2, \frac{1}{2}(1-t_1-t_2) \right) \\ \frac{3}{19} \leqslant t_4 < \frac{1}{2}t_1 }} \psi\left(\frac{n}{\beta p_2 p_3 p_4}, p_4 \right) \\
=&\ S_{B1} - S_{B2} + S_{B3},
\end{align}
where $\beta \sim (2x)^{1-t_1 -t_2 -t_3}$ and $\left(\beta, P(p_{3})\right)=1$. We know that Lemma~\ref{l21} holds for $S_{B1}$ since $t_1 < \frac{8}{19}$ and $t_2 < \frac{9}{38}$. By a trivial argument, we know that $\beta$ must be a prime. Then we know that Lemma~\ref{l21} also holds for $S_{B2}$. If we can group $\{0,2,3,4\}$ (where $0$ represents $\beta$) into $I$ and $J$ satisfy the conditions in Lemma~\ref{l23}, then Lemma~\ref{l21} holds for $S_{B3}$. Working as above, we get a loss from $S_{B}$ of
\begin{equation}
\int_{(t_1, t_2, t_3, t_4) \in U_{B3}} \frac{\omega_1 \left(\frac{t_1 - t_4}{t_4}\right) \omega_1 \left(\frac{1 - t_1 - t_2 - t_3}{t_3}\right)}{t_2 t_3^2 t_4^2} d t_4 d t_3 d t_2 d t_1 < 0.013062,
\end{equation}
where
\begin{align}
\nonumber U_{B3}(t_1, t_2, t_3, t_4) :=&\ \left\{ (t_1, t_2) \in B,\ \frac{3}{19} \leqslant t_3 < \min\left(t_2, \frac{1}{2}(1-t_1-t_2)\right), \right. \\
\nonumber & \quad \{1,2,3\} \text{ cannot be partitioned into } I \text{ and } J \text{ in Lemma~\ref{l23}}, \\
\nonumber & \quad \frac{3}{19} \leqslant t_4 < \frac{1}{2}t_1, \\
\nonumber & \quad \{0,2,3,4\} \text{ cannot be partitioned into } I \text{ and } J \text{ in Lemma~\ref{l23}}, \\
\nonumber & \left. \quad \frac{3}{19} \leqslant t_1 < \frac{1}{2},\ \frac{3}{19} \leqslant t_2 < \min\left(t_1, \frac{1}{2}(1-t_1) \right) \right\}.
\end{align}

For $S_{C}$ we can perform neither a straightforward decomposition nor a role-reversal, hence we need to discard the whole regions. We remark that in \cite{BakerIrving} and \cite{Stadlmann525} Heath-Brown's identity was used to deal with $S_{C}$, but we can not do that here since the corresponding ``Polymath Type-III information'' cannot cover all cases after a Heath-Brown decomposition. Discarding the region gives the losses of
\begin{equation}
\int_{\frac{3}{19}}^{\frac{1}{2}} \int_{\frac{3}{19}}^{\min\left(t_1, \frac{1}{2}(1-t_1)\right)} \mathbbm{1}_{(t_1, t_2) \in C} \frac{\omega \left(\frac{1 - t_1 - t_2}{t_2}\right)}{t_1 t_2^2} d t_2 d t_1 < 0.235134.
\end{equation}

Finally, by combining (4)--(10), the total loss is less than
$$
0.000829 + 0.013062 + 0.235134 < 0.25
$$
and the proof of Theorem~\ref{t1} is completed.

\bibliographystyle{plain}
\bibliography{bib}

\end{document}